\documentclass[10pt]{amsart}

\usepackage{amsmath, amsthm, mathtools, amssymb}
\usepackage{hyperref}

\newtheorem{thm}{Theorem}[section]
\newtheorem{proposition}[thm]{Proposition}
\newtheorem{lemma}[thm]{Lemma}
\newtheorem{corollary}[thm]{Corollary}

\theoremstyle{definition}
\newtheorem{definition}[thm]{Definition}

\theoremstyle{remark}

\newcommand*{\braces}[1]{\left\lbrace #1 \right\rbrace}
\newcommand*{\setof}{\braces}

\newcommand*{\deftobe}{\mathrel{\coloneqq}}

\newcommand{\maps}{\colon}
\newcommand{\st}{\,:\,}

\newcommand{\bndr}{\partial}
\DeclareMathOperator{\conv}{conv}
\newcommand{\intr}{\operatorname{int}}
\newcommand{\relintr}[1]{\operatorname{relint}(#1)}
\DeclareMathOperator{\aff}{aff}

\DeclareMathOperator{\cone}{cone}
\DeclareMathOperator{\lin}{lin}
\newcommand{\dual}[1]{#1^{\ast}}   % dual space
\newcommand{\ddual}[1]{#1^{\ast\ast}}   % double dual space
\newcommand{\dualc}[1]{#1^{\ast}}   % dual cone
\newcommand{\ddualc}[1]{#1^{\ast\ast}}   % double dual cone
\newcommand{\duala}[1]{#1^{\ast}}   % perpendicular affine section

\newcommand{\norm}[1]{\left\Vert #1 \right\Vert}
\newcommand{\abs}[1]{\left\vert #1 \right\vert}

\newcommand{\codim}{\operatorname{codim}}
\newcommand{\coperp}[1]{\prescript{\perp}{}{#1}}
\newcommand{\weakstar}{$\text{weak}^{\ast}$}

\newcommand{\R}{\mathbb{R}}
\renewcommand{\epsilon}{\varepsilon}
\renewcommand{\phi}{\varphi}
\renewcommand{\emptyset}{\varnothing}

\title[Ellipsoidal cones]{%
   Ellipsoidal cones in normed vector spaces
}%

\author[F.~Jafari \and T.~B.~McAllister]{%
   Farhad Jafari \and Tyrrell B. McAllister
}%
\address
{%
   Department of Mathematics \\
   University of Wyoming \\
   Laramie, WY 82071 \\
   USA
}%
\email[Farhad Jafari]{%
   fjafari@uwyo.edu 
}
\email[Tyrrell B. McAllister]{%
   tmcallis@uwyo.edu
}%



\begin{document}

% % elsarticle Title information
% % 
% \begin{frontmatter}
%    \title%
%    {%
%       Ellipsoidal cones in normed vector spaces
%    }%
%    
%    \author[uwyo]{Farhad Jafari}
%    \ead{fjafari@uwyo.edu}
%    
%    \author[uwyo]{Tyrrell B. McAllister\corref{cor}}
%    \cortext[cor]{Corresponding author}
%    \ead{tmcallis@uwyo.edu}
%    
%    \address[uwyo]
%    {%
%       Department of Mathematics \\
%       University of Wyoming \\
%       Laramie, WY 82071 \\
%       USA
%    }%
%    
%    \begin{keyword}
%       ellipsoidal cone \sep
%       ordered normed linear space \sep
%       centrally symmetric convex body.
%       
%       \MSC[2010]{Primary 46B20; Secondary 52A50, 46B40, 46B10.}
%    \end{keyword}
%    
%    
%    \begin{abstract}
%       We give two characterizations of cones over ellipsoids in real
%       normed vector spaces.  Let $C$ be a closed convex cone with
%       nonempty interior such that $C$ has a bounded section of
%       codimension $1$.  We show that $C$ is a cone over an ellipsoid
%       if and only if every bounded section of $C$ has a center of
%       symmetry.  We also show that $C$ is a cone over an ellipsoid if
%       and only if the affine span of $\bndr C \cap \bndr(a - C)$ has
%       codimension $1$ for every point $a$ in the interior of $C$.
%       These results generalize the finite-dimensional cases proved in
%       \cite{JerMcA2013}.
%    \end{abstract}
% \end{frontmatter}

% amsart abstract and \maketitle
% 
\begin{abstract}
   We give two characterizations of cones over ellipsoids in real 
   \linebreak % Needed for amsart
   normed vector spaces.  Let $C$ be a closed convex cone with
   nonempty interior such that $C$ has a bounded section of
   codimension $1$.  We show that $C$ is a cone over an ellipsoid
   if and only if every bounded section of $C$ has a center of
   symmetry.  We also show that $C$ is a cone over an ellipsoid if
   and only if the affine span of $\bndr C \cap \bndr(a - C)$ has
   codimension $1$ for every point $a$ in the interior of $C$.
   These results generalize the finite-dimensional cases proved in
   \cite{JerMcA2013}.
\end{abstract}

\maketitle

\section{Introduction}
In a landmark paper \cite{Rudin}, W.~Rudin and K.~T.~Smith
answered a question of J.~Korevaar by showing that, if $ X $ is a
strictly convex real Banach space of dimension $ n \neq 2 $ and,
for each finite-dimensional subspace $ \pi $ in $ X $, the best
approximation function $ P_\pi $ is linear, then $X $ is a Hilbert
space.  Their theorem led to the following characterization of
ellipsoids: {\it If $K $ is a centrally symmetric compact convex
body in $n$-space (where $n$ is possibly infinite) such that, for
every $ \nu $-dimensional subspace $ \pi $ with $ 0 < \nu < n-1 $,
the union of the tangency sets of all support planes of $K $ which
are translates of $ \pi $ lies in a plane of dimension $n - \nu$,
then $ K $ is an ellipsoid.} This result is a vivid example of how
the characterization of ellipsoids is intimately tied to
characterizing the Banach spaces that are Hilbert spaces.

In \cite{JerMcA2013}, the second author with J.~Jer\'onimo showed
that a fi\-nite-di\-men\-sion\-al pointed cone in which every
bounded section has a center of symmetry is an ellipsoidal cone.
Hence these are exactly the cones over closed unit balls of
Hilbert spaces that have been translated away from the origin.
The primary goal of this paper is to generalize this result to
infinite-dimensional cones.  While this may seem like a result
that would follow from a straightforward induction argument (at
least for a separable Banach space), such an approach is elusive.
We give an affirmative answer to the infinite-dimensional
generalization by carefully using the following fact: \emph{Cones
with bounded sections have dual cones with nonempty interiors}.

Fix a real normed vector space $V$.  A \emph{cone} in $V$ is a
nonempty convex subset $C\subset V$ that is closed under
nonnegative scalar multiplication.  The cone $C$ is \emph{pointed}
if it contains no line through the origin; that is, if $C \cap
(-C) = \setof{0}$.

Given a convex subset $K \subset V$, let $\aff(K)$ denote the
affine span of $K$, and let $\lin(K)$ be the linear span of $K$.
We write $\intr(K)$ for the interior of $K$ in $V$ with respect to
the norm on $V$.  The \emph{relative interior} $\relintr{K}$ and
the \emph{relative boundary} $\bndr K$ are the interior and
boundary, respectively, of $K$ with respect to $\aff(K)$.  The
\emph{cone over} $K$, denoted $\cone(K)$, is the intersection of
all cones containing $K$.  A \emph{section} of $K$ is the nonempty
intersection of $K$ with a closed affine subspace of $V$.  We call
a section $S$ of $K$ \emph{proper} if $S$ has codimension $1$ with
respect to $\aff(K)$.

While the concept of an ellipsoid in finite dimensions is well
known, in infinite dimensional vector spaces this requires careful
definition.  We define a subset $E \subset V$ to be an
\emph{ellipsoid} if, for some $x \in \aff(E)$, there exists an
inner product on the linear space $\aff(E) - x$ such that $E - x$
is the closed unit ball corresponding to this inner product.  An
\emph{ellipse} is a $2$-dimensional ellipsoid.  A cone $C \subset
V$ is \emph{ellipsoidal} if some proper section of $C$ is an
ellipsoid.

A subset $S \subset V$ is \emph{centrally symmetric} if there
exists a point $x \in V$ such that $S - x = -(S - x)$.  In this
case, $x$ is a \emph{center of symmetry} of $S$, and $x$ is the
unique center of symmetry of $S$ if $S$ is nonempty and bounded.

\begin{definition}[CSS Cones]
   \label{defn:CSSCones}
   Let $C \subset V$ be a cone.  We say that $C$
   satisfies the \emph{centrally symmetric sections
   \textup{(}CSS\textup{)} property} if
   \begin{enumerate}
      \item 
      $C$ is closed and $\relintr{C} \ne \emptyset$,
      
      \item  
      there exists a bounded proper section of $C$, and
   
      \item  
      every bounded proper section of $C$ is centrally symmetric.
   \end{enumerate}
   We call a cone with the CSS property a \emph{CSS cone}.
\end{definition}

Our main result is that the CSS cones in $V$ are precisely the
ellipsoidal cones.
\begin{thm}[proved on p.\ \pageref{proof:CSSproof}]
   \label{thm:CSSTheorem}
   Let $C$ be a cone in a normed vector space $V$.  Then $C$ is a
   CSS cone if and only if $C$ is an ellipsoidal cone.
\end{thm}

The proof that finite-dimensional CSS cones are ellipsoidal
appeared in~\cite{JerMcA2013}.  That article also established
another characterization of fi\-nite-di\-men\-sion\-al ellipsoidal
cones: Such a cone is ellipsoidal if and only if it is a so-called
\emph{FBI cone}.

\begin{definition}
   \label{def:FBI}
   Let $C \subset V$ be a cone.  We say that~$C$ satisfies the
   \emph{flat boundary intersections \textup{(}FBI\textup{)}
   property} if
   \begin{enumerate}
      \item 
      $C$ is closed and $\relintr{C} \ne \emptyset$,
      
      \item  
      there exists a bounded proper section of $C$, and
   
      \item  
      for each $a \in \relintr{C}$, some proper section of $C$
      contains $\bndr C \cap \bndr(a - C)$.
   \end{enumerate}
   We call a cone with the FBI property an \emph{FBI cone}.  (See 
   Figure~\ref{fig:FBICone}.)
\end{definition}

\begin{figure}
   \begin{center}
      \includegraphics{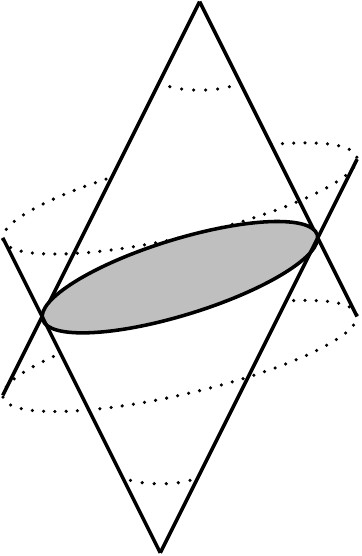}      
   \end{center}
   \caption{An FBI cone $C$ and a translation of $-C$.  The shaded
   region is the convex hull of the intersection of the
   boundaries.  The FBI property implies that this convex hull is
   contained in a hyperplane.}
   \label{fig:FBICone}
\end{figure}

\begin{thm}[proved on p.~\pageref{proof:FBIimpliesEllipsoidal}]
   \label{thm:FBIimpliesEllipsoidal}
   Let $C$ be a cone in a normed vector space $V$.  Then $C$ is a
   FBI cone if and only if $C$ is an ellipsoidal cone.
\end{thm}

Unlike the proof of the CSS characterization of ellipsoidal cones,
the proof of the FBI characterization carries over with very
little change to the infinite-dimensional case.  We give this
proof in Section \ref{sec:FBICones}.  In \cite{JerMcA2013}, the
proof that CSS cones in $\R^{n}$ are ellipsoidal proceeded by
showing that CSS cones are FBI cones, so the proof that FBI cones
are ellipsoidal was key to the argument.  Unfortunately, the proof
in \cite{JerMcA2013} that CSS cones are FBI cones relied on the
existence of a measure, so a different strategy is needed to prove
that CSS cones are ellipsoidal in infinite-dimensional normed
vector spaces.

As a corollary of Theorems \ref{thm:CSSTheorem} and
\ref{thm:FBIimpliesEllipsoidal}, we get two characterizations of
those normed vector spaces that are inner-product spaces.

\begin{corollary}
   Let $V$ be a normed vector space.  Then the following are 
   equivalent.
   \begin{enumerate}
      \item  
      $V$ is an inner product space.
   
      \item
      $V$ contains a CSS cone with nonempty interior.
   
      \item
      $V$ contains an FBI cone with nonempty interior.
   \end{enumerate}
   In particular, if a Banach space $X$ contains a CSS cone or an
   FBI cone, then $X$ is a Hilbert space.
\end{corollary}

\section{Cone lemmas}
\label{sec:ConeLemmas}

In the sections that follow, we will require some lemmas regarding
cones, their duals, their sections, and the relationships between
these concepts and central symmetry.  Most of these lemmas are
well known, though the property of central symmetry seems to be
little-studied in the context of infinite dimensional sections of
cones.  The seminal monograph of Kre\u{\i}n and Rutman \cite{Krein}
still provides an excellent introduction to cones in linear
spaces.  A more recent treatment may be found
in~\cite{AliTou2007}.

Let $V$ be a real normed vector space with norm $\norm{\cdot}$,
and let $V^{*}$ be the dual space of continuous linear functionals
on $V$ under the operator norm, also denoted by $\norm{\cdot}$.
Let a cone $C \subset V$ be given.  Recall that the dual of $C$ is
the cone $\dualc{C} \subset \dual V$ defined by
\begin{equation*}
   \dualc{C} \deftobe \setof{\phi \in \dual V \st \text{$\phi(x)
   \ge 0$ for all $x \in C$}}.
\end{equation*}
Recall also the following well-known result regarding cones in
normed vector spaces.
\begin{proposition}
  Let $C$ be a cone in a normed vector space $V$.  Then, under
  the canonical embedding $V \hookrightarrow \ddual V$, we have
  that $\intr(C) \hookrightarrow \intr(\ddualc C)$.  Moreover,
  the closure of $C$ under this embedding is $\ddualc{C}$.
\end{proposition}

We will find the following notation to be convenient: Given points
$x \in V$ and $\phi \in \dual V$, let $ S_x(\dualc C) $ and $
S_\phi(C) $ be sections of $ C^{*} $ and $ C $, respectively,
defined as follows
\begin{align*}
   S_{x}(\dualc C) & \deftobe \setof{\psi \in \dualc C \st \psi(x)
   = 1},  \\
   S_{\phi}(C) & \deftobe \setof{y \in C \st \phi(y) = 1}.
\end{align*}

More generally, we will occasionally need a canonical affine
subspace in $V$ or $\dual V$ that is perpendicular to a given
affine subspace of $\dual V$ or $V$, respectively.  Recall that,
if $L \subset V$ and $M \subset \dual{V}$ are linear subspaces,
then the \emph{annihilators} of $L$ and $M$ are defined by
\begin{align*}
   L^{\perp} & \deftobe \setof{\phi \in \dual{V} \st
   \text{$\phi(x) = 0$ for all $x \in L$}},\\
   \coperp{M} & \deftobe \setof{y \in V \st \text{$\psi(y) = 0$
   for all $\psi \in M$}}.
\end{align*}
Analogously, given an affine subspace $A \subset V$, respectively
$B \subset \dual{V}$, bounded away from the origin, define the
\emph{perpendicular affine spaces}
\begin{align*}
   A^{\perp} & \deftobe \setof{\phi \in \dual{V} \st
   \text{$\phi(x) = 1$ for all $x \in A$}},\\
   \coperp{B} & \deftobe \setof{y \in V \st \text{$\psi(y) = 1$
   for all $\psi \in B$}}.
\end{align*}

It is well known that, if $L \subset V$ is a closed linear
subspace and if $M \subset \dual{V}$ is a \weakstar-closed linear
subspace, then the annihilators satisfy a duality relation:
$\coperp{(L^{\perp})} = L$ and $(\coperp{M})^{\perp} = M$.

\begin{lemma}
   \label{lem:PerpIsADuality}
   Let $V$ be a normed vector space.  Let $A \subset V$ and $B
   \subset \dual{V}$ be closed affine subspaces such that $0
   \notin A$, $0 \notin B$, and $\codim(A)$ and $\dim(B)$ are
   finite.  Then $\coperp{(A^{\perp})} = A$ and
   $(\coperp{B})^{\perp} = B$.
\end{lemma}

\begin{proof}
   Fix $x_{0} \in A$, $\phi_{0} \in A^{\perp}$, $\psi_{0} \in B$,
   and $y_{0} \in \coperp{B}$.  Let $L \deftobe A - x_{0}$ and $M
   \deftobe B - \psi_{0}$.  Write $x_{0}^{\perp}$ and
   $\coperp{\psi_{0}}$ for the annihilator of the linear span of
   $x_{0}$ and $\psi_{0}$, respectively.  Observe that $A^{\perp}
   = (L^{\perp} \cap x_{0}^{\perp}) + \phi_{0}$ and $\coperp{B} =
   (\coperp{M} \cap \coperp{\psi_{0}}) + y_{0}$.  The equations to
   be proved now follow from the analogous facts about
   annihilators mentioned above.
\end{proof}

\begin{definition}
   A convex set $K \subset V$ is \emph{linearly bounded} if the
   intersection of $K$ with every affine line is a bounded line
   segment.  Equivalently, a linearly bounded convex set contains
   no ray.
\end{definition}

The following lemma establishes a natural relationship between the
non\-empty linearly bounded sections of $ C $ and the dual cone $
C^{*} $.  In particular, the dual cone of $ C$ is the set of
elements of $ \dual V $ that are positive on a linearly bounded
section of $ C $.

\begin{lemma}
   \label{lem:LinearBoundedSectionGeneratesCone}
   Let $C$ be a pointed cone in a normed vector space $V$, and let
   $\phi \in \dual V$ be such that $S_{\phi} \deftobe S_{\phi}(C)$
   is nonempty and linearly bounded.  Then $\dualc C = \setof{\psi
   \in \dual V \st \psi(S_{\phi}) \ge 0}$.
\end{lemma}

\begin{proof}
   We first show that $\phi(C \setminus \setof{0}) > 0$.  For,
   suppose otherwise.  Then, since $C$ is pointed, there exists a
   $v \in C \setminus \setof{0}$ such that $\phi(v) = 0$.  Fix $w
   \in S_{\phi}$.  Then, for every $\lambda \ge 0$, we have that
   $w + \lambda v \in C$ and $\phi(w + \lambda v) = 1$, so $w +
   \lambda v \in S_{\phi}$, which implies that $S_{\phi}$ is not
   linearly bounded.
   
   Now, suppose that $\psi(S_{\phi}) \ge 0$, and let $y \in C$ be
   given.  If $y = 0$, then we immediately have that $\psi(y) \ge
   0$.  If $y \in C \setminus \setof{0}$, then $\phi(y) > 0$, so
   $\frac{1}{\phi(y)}y \in S_{\phi}$.  Thus,
   $\psi(\frac{1}{\phi(y)}y) \ge 0$, so $\psi(y) \ge 0$.
   Therefore, $\psi \in \dualc C$, as claimed.  Since the converse
   claim is immediate, the lemma is proved.
\end{proof}

Furthermore, there is a well-known relationship between the
bounded sections of $ C $ and the interior points of $\dualc C$.
(Cf.~\cite[Theorem 3.8.4]{Jameson}.)  Given a point $v$ in $V$ or
$\dual V$, we write $B_{r}(v)$ to denote the closed ball of radius
$r$ centered at $v$.

\begin{lemma}
   \label{lem:InteriorNormalsGiveBoundedSections}
   Let $C$ be a pointed cone in a normed vector space $V$.  Given
   a nonzero functional $\phi \in \dual V$, the section $S_{\phi}
   \deftobe S_{\phi}(C)$ is bounded if and only if $\phi \in
   \intr(\dualc C)$.  \textup{(}More precisely, given $ r > 0$,
   we have that $S_{\phi} \subset B_{r}(0)$ if and only if
   $B_{1/r}(\phi) \subset \dualc C$.\textup{)}
\end{lemma}

\begin{proof}
   Suppose that $S_{\phi}$ is bounded.  Let $r > 0$ be such that
   $S_{\phi} \subset B_{r}(0)$, and put $\epsilon \deftobe 1/r$.
   Let $\psi \in B_{\epsilon}(\phi)$ be given.  Observe that, for
   all $x \in S_{\phi}$, we have that $\abs{1 - \psi(x)} =
   \abs{\phi(x) - \psi(x)} \le \norm{\phi - \psi} \norm{x} \le
   \epsilon r = 1$, and hence $\psi(x) \ge 0$.  Since $S_{\phi}$
   is linearly bounded, it follows from Lemma
   \ref{lem:LinearBoundedSectionGeneratesCone} that $\psi \in
   \dualc C$, so $\phi \in \intr(\dualc C)$.
   
   Conversely, suppose that $\phi \in \intr(\dualc C)$.  Let
   $\epsilon > 0$ be such that $B_{\epsilon}(\phi) \subset \dualc
   C$, and put $r \deftobe 1/\epsilon$.  Fix $x \in S_{\phi}$.
   Since the unit ball $B_{1}(0)$ has a supporting hyperplane at
   $x / \norm{x}$, there exists a functional $\nu \in \dual V$
   such that $\nu(x / \norm{x}) = 1$ and $\nu(B_{1}(0)) \le 1$.
   That is, $\nu(x) = \norm{x}$ and $\norm{\nu} = 1$.  Set $\psi
   \deftobe \phi - \epsilon \nu$.  Then $\norm{\phi - \psi} =
   \epsilon$, so $\psi \in \dualc C$, and hence $\psi(x) \ge 0$.
   Thus, $\epsilon \norm{x} = \phi(x) - \psi(x) = 1 - \psi(x) \le
   1$, so $\norm{x} \le 1/\epsilon = r$, yielding the claim.
\end{proof}

It is well known that if $ X $ is a separable Banach space, then
every pointed cone $ C $ has a base, i.e.~there is a closed
bounded convex subset $ B \subset X $ such that, for every $ x \in
C\setminus\setof{0} $, there exist unique $ \lambda > 0 $ and $ y
\in B $ such that $ x = \lambda y $.  The separability assumption
is indispensable, as demonstrated by the standard example of the
cone $ C $ of all non-negative real-valued functions on $ X =
\ell^2 (I) $ with respect to the counting measure, where $I$ is an
uncountable set.  For this case, the set of all positive
continuous linear functionals on $ C $ is isometrically isomorphic
to $ C $, but $ C^{*} $ has no strictly positive linear
functionals, and so $ \dualc{C} $ has no interior.  It readily
follows that $ C $ has no base \cite{Jameson}.  It is also worth
noting that pointed cones in separable Banach spaces may not have
a bounded base.  For example, if $ X = C[0,1] $ with its usual
uniform topology, then $ X $ is separable.  If $ C $ is the cone
of nonnegative valued functions in $ X $, then $C$ has a base,
but, since $ C^{*} $ is the set of regular positive Borel measures
on $ [0,1] $, $\dual C$ has an empty interior.  Thus $ C $ has no
bounded base.

We do not assume that our normed vector space $V$ is separable, so
there may exist pointed cones without bounded bases.  However,
condition (2) in the definition of CSS cones (Definition
\ref{defn:CSSCones}) guarantees that this is not the case with a
CSS cone, because having a bounded proper section implies having a
bounded base, as may be seen in the proof of Lemma
\ref{lem:LinearBoundedSectionGeneratesCone}.  In particular, the
dual of a CSS $C$ cone always has a nonempty interior.  It follows
from a series of results due to Borwein and Lewis that the
functionals in the interior of $\dualc C$ are precisely the
functionals that are strictly positive on $C \setminus \setof{0}$.

\begin{lemma}
   \label{lem:BoundedSectionsStrictlySupport}
   Let $C$ be a closed pointed cone in a normed vector space $V$
   such that the dual cone $\dual{C}$ has nonempty interior, and
   let $\phi \in \dual V$.  Then $S_{\phi} \deftobe S_{\phi}(C)$
   is bounded if and only if $\phi$ is strictly positive on
   $C\setminus\setof{0}$.
\end{lemma}
\begin{proof}
   By Lemma \ref{lem:InteriorNormalsGiveBoundedSections},
   $S_{\phi}$ is bounded if and only if $\phi \in
   \intr(\dual{C})$.  Since $\intr(\dual{C}) \ne \emptyset$, it
   follows from \cite[Corollary 2.14]{BorLew1992} and
   \cite[Theorem 3.10]{BorLew1992} that $\phi \in \intr{\dual{C}}$
   if and only if $\phi(C\setminus{0}) > 0$.
\end{proof}

\begin{lemma}
   \label{lem:PerpOfBoundedSectionIsBounded}
   Let $C$ be a pointed cone in a normed vector space $V$.  If
   $\dualc{S}$ is a bounded finite-dimensional section of $\dualc
   C$ that intersects the interior of $\dualc C$, then the section
   \begin{equation*}
      S \deftobe \setof{x \in C \st \text{for all $\phi \in 
      \dualc{S}$, $\phi(x) = 1$}}
   \end{equation*}
   of $C$ is bounded, finite codimensional, and intersects the
   interior of $C$.
\end{lemma}

\begin{proof}
   We first prove that $S \cap \intr(C) \ne \emptyset$.  Let $B =
   \aff(\dualc S)$.  Observe that $S = \coperp{B} \cap C$, and
   suppose, to get a contradiction, that $\coperp{B} \cap \intr(C)
   = \emptyset$.  Then there exists a supporting hyperplane $H$ of
   $C$ containing $\coperp{B}$.  Let $\psi \in \dualc C
   \setminus\setof{0}$ be such that $H = \ker \psi$ and $\psi(C)
   \ge 0$.  Fix $\phi \in \dualc S$.  Then $\phi + \lambda\psi \in
   \dualc C$ for all $\lambda \ge 0$.  Moreover, $\phi + \lambda
   \psi \in (\coperp(B))^{\perp} = B$ by Lemma
   \ref{lem:PerpIsADuality}.  Hence, $\phi + \lambda\psi \in
   \dualc{S}$ for all $\lambda \ge 0$, so $\dualc S$ is not
   bounded, a contradiction.
   
   Let $n \deftobe \dim(\duala S)$.  Since $\duala S$ intersects
   the interior of $\dualc C$, there exist linear functionals
   $\phi_{0}, \dotsc, \phi_{n} \in \intr(\duala S)$ that affinely
   span $\aff(\duala S)$.  It is easy to see that $S =
   \bigcap_{i=0}^{n} S_{\phi_{i}}$.  Indeed, since $S \cap
   \intr(C) \ne \emptyset$, the affine span of $S$ is the 
   intersection of the affine spans of the $S_{\phi_{i}}$:
   \begin{equation*}
      \aff(S) = \bigcap_{i=0}^{n} \aff(S_{\phi_{i}})
   \end{equation*}
   Therefore, $S$ is of co-dimension $\dim(\duala S) + 1$.
   Finally, $S$ is bounded by
   Lemma~\ref{lem:InteriorNormalsGiveBoundedSections}.
\end{proof}

By Lemma \ref{lem:InteriorNormalsGiveBoundedSections}, a bounded
proper section $S$ of a pointed cone $C$ corresponds to an
interior point of $\dualc C$.  This correspondence interchanges
dimension and codimension, since it is the correspondence between
a hyperplane and a vector normal to that hyperplane.  However, if
the bounded proper section $S$ is centrally symmetric, then there
is a canonical corresponding proper section of $\dualc C$, which
is also centrally symmetric.

\begin{lemma}
   \label{lem:NormalOfCentSymmSectionIsCentOfSymm}
   Let $C$ be a pointed cone in a normed vector space $V$.
   Suppose that $x \in C$ and $\phi \in \dualc C$ are such that
   $S_{\phi} \deftobe S_{\phi}(C)$ is bounded and centrally
   symmetric about $x$.  Then $S_{x} \deftobe S_{x}(\dualc C)$ is
   centrally symmetric about $\phi$.
\end{lemma}

\begin{proof}
   Let $\psi \in S_{x}$ be given.  We want to show that $\phi -
   (\psi - \phi) = 2 \phi - \psi \in S_{x}$.  Since it is clear
   that $(2 \phi - \psi)(x) = 1$, it remains only to show that $2
   \phi - \psi \in \dualc C$.  By Lemma
   \ref{lem:LinearBoundedSectionGeneratesCone}, it suffices to
   show that $(2 \phi - \psi)(S_{\phi}) \ge 0$, or, equivalently,
   $\psi(S_{\phi}) \le 2$.  To this end, let $y \in S_{\phi}$ be
   given.  Since $S_{\phi}$ is centrally symmetric about $x$, we
   have that $2x - y \in S_{\phi} \subset C$.  Thus, $\psi(2x - y)
   \ge 0$, or, equivalently, $\psi(y) \le 2$, as desired.
\end{proof}

\section{CSS cones are ellipsoidal cones}
\label{sec:CSSCones}

Fix a CSS cone $C$ with nonempty interior in a normed vector space
$V$.  As mentioned in the introduction, the proof that $C$ is an
ellipsoidal cone in the finite-dimensional case appeared in
\cite{JerMcA2013}.  Somewhat surprisingly, there does not seem to
be a straightforward transfinite-induction argument that extends
this result to the infinite-dimensional case.  We will instead
take a detour through the dual cone $\dualc{C}$.

We call a section $C'$ of $C$ a \emph{sectional subcone} of $C$ if
$C'$ contains the origin.  We will give an argument by finite
induction below showing that every finite-\emph{co}dimensional
sectional subcone of $C$ is also CSS. Had we been able to extend
this induction argument to a transfinite-induction argument, we
could have applied the finite-dimensional CSS characterization of
ellipsoidal cones to prove that every finite-dimensional sectional
subcone of $C$ is ellipsoidal.  The conclusion that $C$ itself is
ellipsoidal would then have followed from the Jordan--von Neumann
characterization of inner-product spaces.

Unfortunately, a direct argument by transfinite induction for the
claim that all sectional subcones of CSS cones are CSS cones
eludes us.  Our strategy instead will be as follows.  To show that
$C$ is ellipsoidal, we will show that its dual $\dualc{C}$ is
ellipsoidal.  That $\dualc{C}$ is ellipsoidal will follow from the
Jordan--von Neumann characterization of ellipsoids once we show
that every finite-dimensional sectional subcone of $\dualc{C}$ is
ellipsoidal.  To prove this, we will need to show that every
bounded finite-dimensional section $\dualc{S}$ of $\dualc{C}$ is
centrally symmetric and then apply the finite-dimensional CSS
characterization of ellipsoidal cones.

Thus, we need to find a center of symmetry $\phi$ for a given
fi\-nite-di\-men\-sion\-al section $\dualc{S}$ of $\dualc C$.  To
do this, we look at the perpendicular section $$S \deftobe
\setof{y \in C \st \text{$\psi(y) = 1$ for all $\psi \in
\dualc{S}$}}$$ of $C$.  It follows from Lemma
\ref{lem:PerpOfBoundedSectionIsBounded} that $S$ is a bounded
section of $C$ with finite codimension.  Our finite induction
argument will thus suffice to show that $S$ is centrally
symmetric, with a center of symmetry $x$.  Furthermore, $S$ is
contained in a proper section $T$ of $C$ with co-dimension $1$,
which will determine a dual vector $\phi \in \dualc{C}$ via
$\phi(T) = 1$.  Finally, the central symmetry of $S$ will imply
that $\dualc{S}$ is centrally symmetric about $\phi$ by Lemma
\ref{lem:NormalOfCentSymmSectionIsCentOfSymm}, establishing the
result.

\begin{lemma}
   \label{lem:FiniteCodimSectionsAreCentSymm}
   Let $C$ be a CSS cone with nonempty interior in a normed vector
   space $V$.  Let $S$ be a bounded section of $C$ such that
   $\codim(S) < \infty$ and $S \cap \intr(C) \ne \emptyset$.  Then
   $S$ is centrally symmetric.
   
   Indeed, $S$ is contained in a proper section $T$ of $C$ such
   that the center of symmetry of $T$ lies on $S$.
\end{lemma}

\begin{proof}
   We begin with the case where $\codim(S) = 2$.  We will show
   that there is a bounded proper section $T$ of $C$ containing
   $S$ whose center of symmetry (which exists by the CSS property)
   lies on $S$.
   
   Fix $y \in S$, and let $L \deftobe \aff(S) - y$.  The image of
   $C$ under the quotient map $Q \maps V \to V/L$ is a
   $2$-dimensional pointed cone in which $Q(S)$ contains a single
   point.  Put $C' \deftobe Q(C)$ and $\setof{s} \deftobe Q(S)$.
   Since quotient maps are open, $s \in \intr(C')$.  Let $\ell
   \subset V/L$ be the affine line through $s$ such that $s$ is
   the midpoint of $\ell \cap C'$.  Let $H \deftobe Q^{-1}(\ell)$,
   and put $T \deftobe H \cap C$.  We claim that $T$ is a bounded
   proper section of $C$ containing $S$ whose center of symmetry
   is on $S$.
   
   It is clear that $S \subset T$.  To see that $T$ is bounded,
   observe that $H - y$ strictly supports $C$ at $0$.  That is,
   $(H-y) \cap C = \setof{0}$.  For, suppose that $z \in (H-y)
   \cap C$.  Then, since $(\ell - s) \cap C' = \setof{0}$, we have
   that $Q(z) = 0$, and so $z \in (S-y) \cap C$.  If $z$ were
   nonzero, then $\lambda z$ would also be in $(S-y) \cap C$ for
   all $\lambda > 0$, so $S$ would be unbounded, contrary to our
   hypothesis.  Hence, $z = 0$.  Thus, by Lemma
   \ref{lem:BoundedSectionsStrictlySupport}, $T$ is bounded, so
   $T$ has a center of symmetry $x$, which must map to the center
   of symmetry of $\ell \cap C'$ under $Q$.  That is, $Q(x) = s$,
   so $x \in S$, as desired.  Thus, $S$ is a section of a
   centrally symmetric set that contains the center of symmetry of
   that set.  Therefore, $S$ itself is centrally symmetric.
   
   If $n \deftobe \codim(S) \ge 3$, fix a codimension-1 linear
   subspace $W \subset V$ containing $S$.  By the preceding
   argument, $C \cap W$ is a CSS cone in which $S$ is a bounded
   codimension-$(n-1)$ section intersecting the interior of $C
   \cap W$.  The theorem now follows from the induction hypothesis
   applied to $C \cap W$.
\end{proof}

\begin{lemma}
   \label{lem:CSSImpliesDualIsCoCSS}
   Let $C$ be a CSS cone with nonempty interior in a normed vector
   space $V$.  Let $S^{\ast}$ be a bounded section of $\dualc C$
   such that $\dim(S^{\ast}) < \infty$ and $S^{\ast} \cap
   \intr(\dualc C) \ne \emptyset$.  Then $S^{\ast}$ is centrally
   symmetric.
\end{lemma}

\begin{proof}
   Let $S \deftobe \setof{y \in C \st \text{$\psi(y) = 1$ for all
   $\psi \in S^{\ast}$}}$.  By Lemma
   \ref{lem:PerpOfBoundedSectionIsBounded}, we have that $S$ is
   bounded, intersects the interior of $C$, and has finite
   codimension.  Thus, by Lemma
   \ref{lem:FiniteCodimSectionsAreCentSymm}, there exists a proper
   section $T$ of $C$ containing $S$ whose center of symmetry $x$
   is in $S$.  Let $\phi \in C^{*}$ be such that $S_{\phi}
   \deftobe S_{\phi}(C) = T$.  Then, by
   Lemma~\ref{lem:NormalOfCentSymmSectionIsCentOfSymm}, $\phi$ is
   the center of symmetry of $S_{x} \deftobe S_{x}(\dualc C)$.
   Note that $S^{\ast} \subset S_{x}$, so it remains only to show
   that $\phi \in S^{\ast}$.
   
   Indeed, since $S \cap \intr C \ne \emptyset$, we have that
   $\aff(S) = \coperp(\aff \dualc S)$.  Thus, $\aff(\dualc S) =
   (\aff S)^{\perp}$ by Lemma \ref{lem:PerpIsADuality}.  In
   particular, $\phi \in \aff(S^{\ast})$.  Since $\phi \in \dualc
   C$, we conclude that $\phi \in S^{\ast}$, as desired.
\end{proof}

The previous Lemma motivates the following definition.

\begin{definition}
   Let $C$ be a closed pointed cone in a normed vector space $V$
   with nonempty interior.  We call $C$ \emph{co-CSS} if every
   bounded finite-dimensional section of $C$ intersecting the
   interior of $C$ is centrally-symmetric.
\end{definition}

Thus, Lemma \ref{lem:CSSImpliesDualIsCoCSS} says that the dual of 
a CSS cone is co-CSS.

\begin{lemma}
   \label{lem:FiniteDimSubConesOfCoCSSAreEllipsoidal}
   Let $C$ be a co-CSS cone in a normed vector space $V$.  Fix a
   finite-dimensional subspace $L \subset V$ such that $L \cap
   \intr C \ne \emptyset$.  Then the cone $L \cap C$ is
   ellipsoidal.
\end{lemma}

\begin{proof}
   Since $C$ is pointed, the finite-dimensional cone $L \cap C$ is
   also pointed.  Since $C$ is co-CSS, every bounded proper
   section of $L \cap C$ is centrally symmetric.  In particular,
   $L \cap C$ is CSS. Therefore, by the finite-dimensional CSS
   characterization of ellipsoidal cones \cite[Theorem
   1.4]{JerMcA2013}, $L \cap C$ is ellipsoidal.
\end{proof}

\begin{lemma}
   \label{lem:DualsOfCSSConesAreEllipsoidal}
   Let $C$ be a CSS cone in a normed vector space $V$.  Then the
   dual cone $\dualc C$ is ellipsoidal.
\end{lemma}

\begin{proof}
   Since $C$ is CSS, there exists a bounded proper section
   $S_{\phi} \deftobe S_{\phi}(C)$ of $C$, where, by Lemma
   \ref{lem:InteriorNormalsGiveBoundedSections}, $\phi \in
   \intr(\dualc C)$.  Let $x \in C$ be the center of symmetry of
   $S_{\phi}$.  Then, by Lemma
   \ref{lem:NormalOfCentSymmSectionIsCentOfSymm}, the section
   $S_{x} \deftobe S_{x}(\dualc C)$ is centrally symmetric about
   $\phi$.  In addition, $S_{x}$ is a proper section of $\dualc C$
   because it contains the point $\phi \in \intr(\dualc C)$.
   Furthermore, $S_{x}$ is bounded by Lemma
   \ref{lem:InteriorNormalsGiveBoundedSections} because $x \in
   \intr(C) \hookrightarrow \intr(\ddualc C)$ under the canonical
   embedding.  Finally, by Lemma
   \ref{lem:FiniteDimSubConesOfCoCSSAreEllipsoidal}, every
   finite-dimensional section of $S_{x}$ through $\phi$ is an
   ellipsoid centered at $\phi$.  Hence, by the Jordan--von
   Neumann characterization of inner-product spaces
   \cite{JorvNeu1935}, $S_{x}$ is an ellipsoid.  Therefore, $\dual
   C$ is ellipsoidal.
\end{proof}

We are now ready to prove our main result.

\begin{proof}[Proof of Theorem \ref{thm:CSSTheorem}]
   \label{proof:CSSproof}
   Every bounded proper section of an ellipsoidal cone is an 
   ellipsoid, and ellipsoids are centrally symmetric.  Hence, 
   ellipsoidal cones are CSS.
   
   To prove the converse, let a CSS cone of dimension $\ge 2$ be
   given.  By Lemma \ref{lem:DualsOfCSSConesAreEllipsoidal},
   $\dual{C}$ is ellipsoidal.  Fix an ellipsoidal proper section
   $\dual{S}$ of $\dual{C}$, and let $\phi \in \dual{S}$.  Thus,
   we have an inner product on the codimension-$1$ linear subspace
   $M \deftobe \aff(\dual{S}) - \phi$.  Since $\dual{S}$ is
   bounded, $\aff(\duala{S})$ does not contain the origin.  Hence,
   we can complete the inner product on $M$ to an inner product on
   all of $\dual{V}$.  Indeed, since $\dual{V}$ is already a dual
   space, it is in fact a Hilbert space.  The dual of an
   ellipsoidal cone in a Hilbert space is ellipsoidal
   \cite[p.~51]{Krein}, so $\ddual{C}$ is ellipsoidal.  Since
   $\ddual{C}$ is the closure of $C$ under the canonical embedding
   $V \hookrightarrow \ddual{V}$, we conclude that $C$ itself is
   ellipsoidal in $V$.
\end{proof}

It follows that if a normed vector space $V$ contains a CSS cone
with nonempty interior, then $V$ is an inner product space.  In
particular, if a Banach space $X$ contains a full-dimensional CSS
cone, then $X$ is a Hilbert space.

\section{FBI cones are ellipsoidal cones}
\label{sec:FBICones}

We conclude by proving that every cone in a normed vector space
$V$ that satisfies the FBI property (Definition \ref{def:FBI}) is
ellipsoidal.

\begin{lemma}
   \label{lem:ConeIntervalsAreBounded}
   Let $C$ be a cone in a normed vector space $V$ such that
   $\intr(C)$ and $\intr(\dualc C)$ are both nonempty.  Then, for
   each $a \in \intr(C)$, the intersection $C \cap (a - C)$ is
   bounded.
\end{lemma}

\begin{proof}
   Since $\intr(\dualc{C}) \ne \emptyset$, there exists a
   functional $\phi \in \dualc C$ such that $S_{\phi} \deftobe
   S_{\phi}(C)$ is a bounded base of $C$ by Lemma
   \ref{lem:LinearBoundedSectionGeneratesCone}.  In particular,
   $\ker(\phi)$ strictly supports $C$ at $0$ by Lemma
   \ref{lem:BoundedSectionsStrictlySupport}.  By suitably
   normalizing $\phi$, we also have that $S_{\phi}$ strictly
   supports $a - C$ at $a$.  Thus,
   \begin{equation*}
      C \cap (a - C) \subset \setof{y \in C \st \phi(y) \le 1}.
   \end{equation*}
   Since $S_{\phi}$ is bounded and is a base for $C$, it follows
   that $C \cap (a - C)$ is also bounded.
\end{proof}

\begin{proof}[Proof of Theorem \ref{thm:FBIimpliesEllipsoidal}]
   \label{proof:FBIimpliesEllipsoidal}%
   Without loss of generality, suppose that $\intr{C} \ne
   \emptyset$.  Fix $x \in \intr(C)$, and set $S \deftobe
   \conv(\bndr C \cap \bndr(2x - C))$.  By the FBI property and
   Lemma \ref{lem:ConeIntervalsAreBounded}, $S$ is a bounded
   section of $C$.  Moreover, $S$ is centrally symmetric about
   $x$.  We show that every $2$-dimensional section of $S$ through
   $x$ is an ellipse.  Let $E$ be such a section.  Observe that $0
   \notin \aff E$, because $0 \in \aff E \subset \aff S$ would
   imply that $\aff S = \lin S$, contrary to the boundedness of
   $S$.  Therefore, $C' \deftobe \cone(E)$ is a $3$-dimensional
   cone.  We claim that $\cone(E)$ is an FBI cone.  To prove this,
   let $y \in \relintr C'$, and let $\Gamma \deftobe \bndr(C')
   \cap \bndr(y - C')$.  On the one hand, $\Gamma$ is contained in
   $\lin E$.  On the other hand, let $H$ be the affine hyperplane
   in $V$ containing $\bndr C \cap \bndr(y - C)$.  As above, $0
   \notin H$.  Since $0 \in \lin E$, it follows that $\Gamma
   \subset (\lin E) \cap H$, but $\lin E \not\subset H$.  Hence,
   $\Gamma$ is contained in a $2$-dimensional affine subspace of
   $\lin E$.  That is, $C'$ is an FBI cone.  It follows from
   \cite[Theorem 1.2]{JerMcA2013} that $E$ is an ellipse.
   Therefore, by the Jordan--von Neumann characterization of
   inner-product spaces \cite{JorvNeu1935}, $S$ is an ellipsoid.
\end{proof}

% \bibliography{Bibliography}
% \bibliographystyle{amsplain-fi-arxlast}

\providecommand{\bysame}{\leavevmode\hbox to3em{\hrulefill}\thinspace}
\providecommand{\MR}{\relax\ifhmode\unskip\space\fi MR }
\providecommand{\MRhref}[2]{%
  \href{http://www.ams.org/mathscinet-getitem?mr=#1}{#2}
}
\providecommand{\href}[2]{#2}

\end{document}